\documentclass[11pt,oneside]{article}
\usepackage{amsthm, amssymb, amsmath,psfrag}
\usepackage{pstool}
\usepackage{graphicx}
\usepackage{epsfig}
\usepackage{hyperref}

\newcommand{\norm}[1]{\left\Vert#1\right\Vert}
\newcommand{\abs}[1]{\left\vert#1\right\vert}

\theoremstyle{plain}
\newtheorem{thm}{Theorem}[section]
\newtheorem{cor}[thm]{Corollary}

\newtheorem{prop}[thm]{Proposition}

\theoremstyle{definition}
\newtheorem{defn}[thm]{Definition}

\theoremstyle{remark}

\newtheorem{rem}{Remark}

\numberwithin{equation}{section}

\newcommand{\dist}{\operatorname{dist}}

\title{On non-positive curvature properties of the  Hilbert metric}

\author{Layth M. Alabdulsada and L\'{a}szl\'{o} Kozma}

\begin{document}
\maketitle
\begin{abstract}
In this paper, we consider different types of non-positive curvature properties of the Hilbert metric of a convex domain in ${\Bbb R^n}$.
First, we survey the relationships among the concepts and prove that in the case of Hilbert metric some of them are equivalent.
Furthermore, we show some condition which implies the rigidity feature: if the Hilbert metric is Berwald, i.e., its Finslerian Chern connection reduces to a linear one, then the domain is an ellipsoid and the metric is Riemannian.
\end{abstract}

{\small
\noindent
\medskip\noindent{\sl Subjclass:\/} 53C60.
}

\section{Introduction}

The Hilbert geometry of a convex domain is just the generalization the Cayley-Klein model of the hyperbolic plane, first introduced by D. Hilbert in 1908.
The Hilbert metric is not only an example of Finsler metric with special metric but also a geodesic metric space, where several synthetic concepts of hyperbolicity, i.e., of non-positive curvature can be used.
 Our aim is to find what consequences are implied by the non-positive curvature properties of the Hilbert metric, and whether the mutual relationships of the different non-positivity concepts reduce to each other in this case.

In geodesic metric spaces, several types of non-positivity concepts were introduced by Alexandrov, see \cite{Jost},
Busemann \cite{Busemann}, and others. In the general context, the Alexandrov's one is the strongest, implying Busemann's one.
Then a weaker class of non-positively curved spaces is the spaces of peakless metrics,
called Pedersen type metrics by some authors, and the weakest one is the class of geodesic metric spaces with convex capsules. (See the definitions in Section 2).
Examples show that the inclusion among these classes is proper.

Turning to the Hilbert geometry, it was proved (\cite{Busemann}) that
for any strictly convex domain, the Hilbert metric automatically satisfies the 2 last mentioned assumptions of non-positivity,
namely, the Pedersen property and having convex capsules.
Alexandrov's and Busemann's assumptions are more rigid \cite{K-S}: if the Hilbert metric of a convex domain satisfies the Busemann (or Alexandrov)
nonpositivity property, then the domain is an ellipsoid, and the metric is a Riemannian one,
and equivalent to the Cayley-Klein model of hyperbolic plane.

During the more than one-hundred years elapsed, a plenty of geometrical properties of the Hilbert metric have been investigated from several aspects.
The reader may find a comprehensive survey of its history in \cite{Papa} by  Papadopoulos.

It has been turned out that the Hilbert metric is also a special case of Finsler geometry, which was introduced later (by P. Finsler in 1917),
and deeply analyzed by L. Berwald in the twenties-thirties of the last century. Nevertheless, about 60 years later T. Okada (\cite{Okada}) proved in a transparent manner that the Hilbert metric is projectively flat, and its flag curvature is negative constant.

In the sequel, we will present 4 concepts which are linked in our investigations. First, in Section 2 we give the basic definitions of 4 types of non-positive curvature in geodesic metric spaces, some of their properties and their general relationships.
After defining the Hilbert metric of a convex domain in Section 3, and listing some results important for the next steps, we show that in the case of Hilbert metric, the first two concepts of non-positivity are equivalent if and only if the domain is an ellipsoid.
Stepping to the analytical considerations in Section 4 we give some basics of Finsler geometry, and especially, our special case called Berwald space.
Then we present those relationships and known results about these concepts which are used to prove our result. Finally, we show that if the Finsler metric induced by the Hilbert metric of a convex domain is Berwald, then the domain should be an ellipsoid.

\section{Non-positive curvature  concepts in metric spaces}

Let $(M,d)$ a metric space, and $\gamma :[0,1] \rightarrow M$ a curve in $M$. The length of $\gamma$ is
$$ \ell(\gamma):= \mbox{sup} \left \{ \sum_{i=1}^{n } d(\gamma (t_{i-1}),\gamma(t_i) ): 0 = t_0 \textless t_1 \textless ... \textless t_n =1, n \in \mathbb N \right \}.$$
A curve $ \gamma :[0,1] \rightarrow M $
is called a \emph{geodesic} if there exists $ \epsilon  > 0 $ such that $$\ell(\gamma_{\restriction{[t_1,t_2 ]}}) = d(\gamma (t_1),\gamma(t_2) )  \mbox {\     whatever}   \abs{t_1-t_2}  \textless \ \epsilon, t_1,t_2 \in [0,1]. $$
This property is independent of the choice of parametrization, although the value of $\epsilon$ may change.
A geodesic $\gamma :[0,1] \rightarrow M$  is called a \emph{shortest geodesic}  if$$ \ell(\gamma) = d ( \gamma (0), \gamma (1)).$$

\begin{defn}  {\cite{Jost}}(Geodesic Length Space)%
{\ A metric space $(M,d)$ is called a \emph{geodesic length space}, or simply a \emph{geodesic space}, if for any two points $x,y \in M $ there exists a shortest geodesic  joining $x$ and $y$.  }
It is called \emph{locally geodesic space} if this property holds in an appropriate neighborhood of any point.\end{defn}

{For $x,y \in M$, we call $ m(x,y) $ a\emph{ midpoint} of $x$ and $y$ if $$ m(x,y)= \gamma (\frac{1}{2})$$ for a shortest geodesic $\gamma :[0,1] \rightarrow M$ from $x$ to $y$, where $\gamma$ is supposed to be parametrized proportionally to arc-length.}

\begin{defn} {\cite{Jost}}
(Alexandrov Non-positive Curvature){ \ A locally geodesic  space $(M,d)$ is said to be an \emph{Alexandrov non-positive curvature  space }if for every ${p\in M}$ there exists $\delta_p > 0$ such that for every $ x,y,z \in B(p,\delta_p)$ and any shortest geodesic $  \gamma:[0,1] \rightarrow M $ with $ \gamma (0)=x,  \gamma(1)=z$, we have for $ 0\leq t \leq 1 $
\begin{equation*}d^2(y,  \gamma(t))\leq (1-t) \ d^2(y, x)+t \ d^2(y, z)-t(1-t) \ d^2(x,z).\end{equation*}
(Alexandrov non-positive curvature inequality).}
\end{defn}

\begin{rem}
In some literature, the Alexandrov non-positive curvature  spaces are called CAT(0)--spaces as well. Furthermore, the complete CAT(0)--spaces are the Hadamard manifolds, see {\cite{B-H-J-K-S}}.
\end{rem}

\begin{defn} {\cite{Jost}}
(Busemann Non-positive Curvature) { \ A locally geodesic space $(M,d)$ is said to be a \emph{Busemann non-positive curvature space}  if for every ${p\in M}$ there exists $\delta_p > 0$ such that for all $ x,y,z \in B(p,\delta_p)$
 \begin{equation*}d(m(x,y),m(x,z))\leq \frac{1}{2}d(y,z).\end{equation*} (Busemann non-positive curvature inequality).}
\end{defn}
In other words, for any two shortest geodesic $  \gamma_1,  \gamma_2:[0,1] \rightarrow M $, with $   \gamma_1(0) = x = \gamma_2(0) ) \in B(p,\delta_p )$ and with endpoints of $\gamma_1,  \gamma_2 \in B(p,\delta_p )$, we have
\begin{equation*}d ( \gamma_1(\frac{1}{2}), \gamma_2(\frac{1}{2}))\leq \frac{1}{2}d( \gamma_1(1), \gamma_2(1)).\end{equation*}

Now, define the distance of a curve $\gamma$ and a point $q\in M$ as
$$ \dist (\gamma, q) = \inf \{ d(\gamma(t), q) \,:\, 0\le t\le 1 \}.   $$

\begin{defn} { \cite{K-K}}
(Pedersen  Non-positive Curvature) {A locally geodesic space is said to be a\emph{ Pedersen  non-positive curvature space} if for every ${p\in M}$ there exists $\delta_p > 0$ such that for any two shortest geodesic $\gamma_1,  \gamma_2: [0,1] \rightarrow B(p,\delta_p )$ the function $ f :[0,1] \rightarrow \Bbb R$, defined by
\begin{equation*} f(t)= \dist(\gamma_1,  \gamma_2(t))\end{equation*}
is quasiconvex, i.e., for every $t \in [0,1]$, $ f(t) \leq   \max \{f(0), f(1)\}$.}

 Let $\gamma: [a,b] \rightarrow M$  be a shortest geodesic and $\alpha > 0$. Attached to $\gamma$ and $\alpha$, we define the \emph{capsule} as   \begin{equation*}\mathcal {C_{\gamma}}(\alpha)= \{q \in M: \dist(\gamma, q) \leq \alpha \}.\end{equation*}
Let $M_0$ be a non-empty subset of $M$. The pair $(\gamma, \alpha)$ is said to be $M_0$-admissible if $\mathcal {C_{\gamma}}(\alpha) \subset M_0$.
\end{defn}

\begin{defn}{ \cite{K-K}}
(Convex Capsules) {We say that a locally geodesic space $(M, d)$ has \emph{convex capsules} if for every ${p\in M}$ there exists $\delta_p > 0$ such that for every $B(p,\delta_p )$-admissible pair $(\gamma, \alpha)$, the capsule $\mathcal {C_{\gamma}}(\alpha)$ is convex.}
\end{defn}

\begin{rem}{In general, an Alexandrov non-positive curvature space is a Busemann non-positive curvature space (see \cite{Jost}, Chapter 2). On the other side, some initiative examples show that not all Busemann non-positive curvature spaces are Alexandrov non-positive curvature spaces \cite{Jost}}.
\end{rem}
\begin{defn}{\cite{Shijie}}
(Average Angle){ Let $\gamma_1: [0,a] \rightarrow M$ and $\gamma_2: [0,b] \rightarrow M$ be two shortest geodesics  with $p= \gamma_1(0)= \gamma_2(0).$
\emph{The average angle} between  $\gamma_1$ and  $\gamma_2$ at $p$ is defined by
$$\measuredangle (\gamma_1, p, \gamma_2) = \lim_{n \to \infty}A_{\gamma_1, \gamma_2}(\frac{a}{2^n}, \frac{b}{2^n}), $$
if the limit of the sequence exists, where the comparison angle is given by
$$A_{\gamma_1, \gamma_2}(a,b) := \text{arccos} \ \frac{a^2+b^2- d(\gamma_1(a), \gamma_2(b))^2}{2ab}.$$
Let $q$ be an inner point of a shortest geodesic $pr$, and  $qs$ be a shortest geodesic. It is clear that for an Alexandrov non-positive curvature space
the sum of adjacent average angles is at least $\pi$, i.e., $\measuredangle (p,q, s)+ \measuredangle (s,q, r) \geq \pi$.
}
\end{defn}

\begin{thm}{\rm \cite{Shijie}} In a locally geodesic space the Alexandrov and  Busemann non-positive curvature properties are equivalent if and only if the sum of adjacent average angles $\geq \pi$.
\end{thm}

\section{The Hilbert metric of a convex domain and its curvature}

\begin{defn} \label{H-metric} Let $K$ be a bounded convex open set in ${\Bbb R^n}$ $( n\geq 2)$. The Hilbert metric $d_K$  on $K$ is defined as follows. For any $x\in K$, let $d_K(x,x)=0$. For distinct points $x,y \in K $,  assume that the straight line passing through $x,y$ intersects the boundary $\partial  K$ at two points $a,b$ such that the order of these four points on the line is $a,x,y,b$ as in Figure 1.

 {Denote the cross-ratio of the points by $$ [a,x,y,b]= \frac {\norm{y-a}}{\norm {y-x}}  \frac {\norm {b-x}}{\norm{ b-a}}.$$}
 {where $\norm{.}$ is the Euclidean norm of ${\Bbb R^n}$. Then the \emph{Hilbert metric }is
 \begin{equation*}d_K(x,y)=\frac{1}{2} \ln [a,x,y,b], \end{equation*}
 and the metric space $(K,d_K)$ is called a\emph{ Hilbert geometry}.}
\begin{figure}[h]
\psfrag{a}{$a$}
\psfrag{b}{$b$}
\psfrag{x}{$x$}
\psfrag{y}{$y$}

\begin{center}
  \centering \includegraphics[scale=0.6]{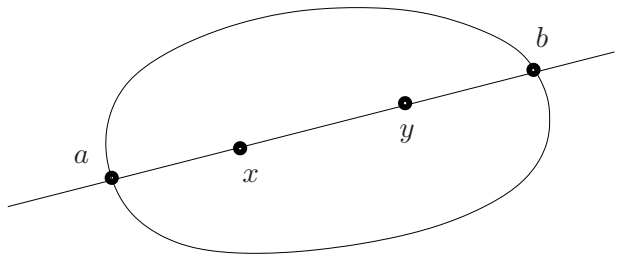}
\caption{ Hilbert geometry }
\end{center}
\end{figure}
\end{defn}

Concerning the curvature properties of the Hilbert metric, Busemann (\cite{Busemann}, page 108) showed that for any strictly convex domain $K$, the Hilbert metric $d_K$ satisfies the Pedersen non-positivity curvature property, and, consequently,  has a convex capsules.

{Kelly and Straus published two papers on the curvature of Hilbert geometry in 1958 and 1968. They used the concept of Busemann non-positive curvature, which is a pleasant geometric approach and weaker than Alexandrov's one. Nevertheless in the case of Hilbert metric of a convex domain it has a strong consequence, namely it implies the reduction of the domain to an ellipsoid. In details, it was proved in \cite{K-S}:}
\begin{prop}
If  Hilbert metric $(K,d_K)$ has  Busemann non-positive curvature, then the domain $K$ is an ellipsoid and the Hilbert metric $d_K$ is hyperbolic, i.e., Riemannian.
\end{prop}

\begin{cor} {\rm  (see \cite[Corollary 5.6]{Guo})}
A Hilbert metric $(K,d_K)$  satisfies the Alexandrov non-positive curvature condition if and only if $K$ is an ellipsoid.
\end{cor}

\begin{prop}
Let $(K,d_K)$  be the Hilbert metric of a convex domain $K$. Then the Busemann non-positive curvature is equivalent to the Alexandrov non-positive curvature.
\end{prop}
\begin{proof}
First, it is true in general that all Alexandrov non-positive curvature is Busemann non-positive curvature (\cite{Jost}). Conversely, if the Hilbert metric is Busemann non-positive curvature, then by a paper of Kelly and Straus proved in 1958 (\cite{K-S}) the domain is an ellipsoid, and the metric is a Riemannian. The latter property, however, implies by the above corollary that the Alexandrov non-positive curvature property is also satisfied.
\end{proof}

\begin{cor}
In Hilbert geometry, the sum of adjacent average angles is $\geq \pi.$
\end{cor}
\begin{proof}
Gu proved in \cite{Shijie} that the sum of adjacent average angles is $\geq \pi$ if and only if the Alexandrov and  Busemann non-positive curvature properties are equivalent. Conversely, by our proposition, we get the corollary immediately.
\end{proof}

\section{Finsler structure of the Hilbert metric }

 {In this section, we recall briefly some
known facts about Finsler and Berwald spaces. For details, see \cite{B-C-S}.}

 {Let $M$ be an $n$-dimensional $C^{\infty}$ manifold and
$TM=\bigcup_{x \in M}T_x M $ the tangent bundle. If the continuous
function $ F:TM\to R_+ $ satisfies the conditions that it is
$C^\infty$ on $TM\setminus \{ 0 \};$ $ F(tu)=tF(u)$ for all $t\geq
0$ and  $u\in TM,$ i.e., $F$ is positively homogeneous of degree
one; and the matrix $g_{ij}(u):=(\frac 12F^2)_{y^iy^j}(u)$ is
positive definite for all $u\in TM\setminus \{ 0 \},$ then we say
that $F$ is a Finsler fundamental function and $(M,F)$ is a {\it Finsler manifold}.}

Every Finsler fundamental function naturally determines a metric $d_F$ as follows:
 {Let $\gamma : [0,r]\to M$ be a piecewise $C^\infty$ curve. Its
{\it integral length} is defined as
\begin{equation*} L(\gamma)= \int_0^r
F(\gamma(t), \dot\gamma(t))\,dt. \end{equation*}
For $x_0,x_1\in M$ denote by
$\Gamma(x_0,x_1)$ the set of all piecewise $C^\infty$ curves
$\gamma:[0,r]\to M$ such that $\gamma(0)=x_0$ and $\gamma(r)=x_1$.
Define a map $d_F: M\times M \to [0,\infty)$ by $$d_F(x_0,x_1) =
\inf_{\gamma \in \Gamma(x_0,x_1)} L(\gamma).$$ Of course, we have
$d_F(x_0,x_1)\geq 0,$ where equality holds if and only if
$x_0=x_1;$ and $d_F(x_0,x_2) \leq d_F(x_0,x_1)+d_F(x_1,x_2).$  In
general, since $F$ is only a positive homogeneous function,
$d_F(x_0,x_1)\neq d_F(x_1,x_0)$, therefore $(M,d_F)$ is only a
non-reversible metric space, in general.}

Let $\pi^*TM$ 
be
the pull-back of the tangent bundle $TM$ by
$\pi:TM\setminus\{0\}\to M.$  Unlike the Levi-Civita connection in
Riemann geometry, there is no unique natural connection in the
Finsler case. Among these connections on $\pi^*TM,$ we choose the
{\it Chern connection}  whose coefficients are denoted by
$\Gamma^i_{jk}$  (see \cite[p.~38]{B-C-S}). This connection induces
the {\it hh-curvature tensor}, denoted by $R$ (see \cite[Chapter
3]{B-C-S}).

{Let $(x,y)\in TM\setminus 0$ and $V$ a section of
the pulled-back bundle $\pi^*TM$. Then\\
\begin{equation*} 
\kappa(y,V) =\frac {g (R(V,y)y, V)} {g (y,y) g (V,V)
- [g (y,V)]^2},
\end{equation*}
is the {\it flag curvature} with flag $y$ and transverse edge $V$.
Here $$g_{(x,y)}:=g_ {i j}  dx^i \otimes dx^j := (\frac
12F^2)_{y^iy^j} dx^i\otimes dx^j$$ is the Riemannian metric on the
pulled-back bundle $\pi^*TM$ (see \cite[p.~68]{B-C-S}). When $F$ is
Riemannian, then  the flag curvature coincides with the sectional
curvature.  Let $\kappa$ abbreviate the collection of flag curvatures
$$\{\kappa(V,W):   0\not= V,W \in T_xM, x\in M, \mbox{$V$ and $W$ are
not collinear}\}.$$ We say that the flag curvature of $(M,F)$ is
{\it non-positive} if $\kappa\le 0.$}

{ A Finsler manifold is of
{\it Berwald type}
if the Chern connection coefficients
$\Gamma_{ij}^k$ in natural coordinates depend only on the base
point (see \cite[p.~258]{B-C-S}).}

{Krist\' {a}ly et al.} proved in \cite{K-K} that all mentioned non-positivity properties are equivalent to the analytical condition $ \kappa  \leq 0 $ in the case of Berwald space.

\begin{thm}{\rm \cite{K-K}} Let $(M,F)$ be a Berwald space where $F$ is positively (but perhaps not absolutely) homogeneous of degree one.The following assertions are equivalent:
\begin{itemize}
\item [a)] The flag curvature $ \kappa $ of $(M, F)$ is non-positive;
\item[b)] $(M, d_F)$ is a Busemann non-positive curvature space;
\item[c)] $(M, d_F)$ is a forward Pedersen non-positive curvature space;
\item [d)]$(M, d_F)$ is a backward Pedersen non-positive curvature space;
\item [e)]$(M, d_F)$ has  convex  forward capsules;
\item[f)] $(M, d_F)$ has  convex  backward capsules.
\end{itemize}
\end{thm}

The Hilbert metric $d_K$ of the convex open domain $K$ naturally determines its Hilbert Finsler fundamental function $F_K$ as follows (\cite{T}:
First the asymmetric Finsler metric, called Funk metric $\widetilde F_K$ is defined by
$$ p + \dfrac{1}{\widetilde F_K(u)} u \in \partial K \qquad \mbox{for any}\  u\in T_pK, \mbox{and} \ p\in K, $$
and then $F_K$ is obtained by symmetrization:
$$ F_K(u)= \frac 12 (\widetilde F_K (u) + \widetilde F_K(-u)).$$
Naturally, $F_K$ is a reversible Finsler metric, therefore the forward and backward concepts coincide.
It is easy to check that the induced distance of $F_K$ is just the Hilbert distance $d_K$ defined above in Definition \ref{H-metric}.

The flag curvature of the Hilbert metrics was computed in 1929 by Funk in dimension 2 and by Berwald in all dimensions. Later T.~Okada proposed a more direct computation:
\begin{thm}{\rm \cite{Okada}} The Hilbert metric $(K,d_K)$ is projectively flat Finsler space of negative constant curvature $-1$.
\end{thm}

\begin{prop}
If the Hilbert metric $d_K$ of a convex domain $K$ is a Berwald metric, then it reduces to a Riemannian metric, and the domain is an ellipsoid.

\end{prop}

\begin{proof}
By a theorem of Okada (\cite{Okada}) the flag curvature is negative constant for the Hilbert metric of a convex domain. Krist\' {a}ly and Kozma showed in \cite{K-K}, among others, that for any Berwald space the non-positivity of the flag curvature is equivalent to the property of Busemann non-positive curvature.
Moreover, Kelly and Straus proved in 1958 (\cite{K-S}) that if the Hilbert metric satisfies Busemann's non-positive curvature property, then the domain is an ellipsoid, and the metric is a Riemannian one.
\end{proof}
From our proposition, we get immediately the next corollary, because all Riemannian metric is Berwald.

\begin{cor}  {\rm  (see \cite[Theorem 11.6]{T})}The Hilbert metric $d_K$ of a bounded convex domain $ K \subset {\Bbb R^n}$ with smooth strongly convex boundary is Riemannian if and only if   $K$  is an ellipsoid.
\end{cor}

\begin{rem}
Conversely, if the domain is not an ellipsoid, then $d_K$ is a non-Berwaldian projectively flat metric.
\end{rem}

{
  \bigskip
  \footnotesize

  Layth M. Alabdulsada, \textsc{ Institute of Mathematics, University of Debrecen, H-4002 Debrecen, P.O. Box 400, Hungary}\par\nopagebreak
  \textit{E-mail address:\ }\texttt{layth.muhsin@science.unideb.hu}

  \medskip

  L\'aszl\'o  Kozma, \textsc{ Institute of Mathematics, University of Debrecen, H-4002 Debrecen, P.O. Box 400, Hungary}\par\nopagebreak
  \textit{E-mail address:\ }\texttt{kozma@unideb.hu}
}

\end{document}